\documentclass[10pt, english]{elsarticle}
\usepackage{amsthm}
\usepackage{amsmath}
\usepackage{latexsym, amssymb}
\usepackage{txfonts}
\usepackage{mathtools}
\usepackage{color}
\usepackage{babel}
\usepackage[all]{xy}
\usepackage[capitalise]{cleveref}

\newtheorem{thm}{Theorem}[section] 

\newtheorem{cor}[thm]{Corollary}

\newtheorem{lem}[thm]{Lemma}
\newtheorem{prop}[thm]{Proposition}

\theoremstyle{definition}
\newtheorem{rem}[thm]{Remark}

\newcommand\operA[2]{{\if!#2!\operatorname{#1}\else{\operatorname{#1}_{#2}^{\phantom{I}}}\fi}} 

\def\tr{{\operatorname{Tr}}}
\def\dlog{{\operatorname{dlog}}}

\def\norm{{\operatorname{N}}}

\def\s{\sigma}

\newcommand{\Trace}[1][]{\if!#1!\operatorname{Tr}\else{\operatorname{Tr}_{#1}^{\phantom{I}}}\fi} 

\long\def\forget#1\forgotten{{}} %

\def\({\left(}
\def\){\right)}

\newcommand\LAY[3][]{{\begin{array}{c}\mbox{#2} \if#1!{}\else{+}\fi \\ \mbox{#3}\end{array}}}

\makeatletter
\newcommand{\bigperp}{%
  \mathop{\mathpalette\bigp@rp\relax}%
  \displaylimits
}

\newcommand{\bigp@rp}[2]{%
  \vcenter{
    \m@th\hbox{\scalebox{\ifx#1\displaystyle2.1\else1.5\fi}{$#1\perp$}}
  }%
}
\makeatother

\renewcommand{\geq}{\geqslant}
\renewcommand{\leq}{\leqslant}

\makeatletter
\def\ps@pprintTitle{%
 \let\@oddhead\@empty
 \let\@evenhead\@empty
 \def\@oddfoot{\centerline{\thepage}}%
 \let\@evenfoot\@oddfoot}
\makeatother

\newif\iffurther
\furtherfalse

\input xy
\input xyidioms.tex
\usepackage{xy}
\xyoption{all} %

\begin{document}
\begin{frontmatter}

\title{Linkage of Symbol $p$-Algebras of Degree 3}

\author{Adam Chapman}
\ead{adam1chapman@yahoo.com}
\address{Department of Computer Science, Tel-Hai Academic College, Upper Galilee, 1220800 Israel}

\begin{abstract}
Given a field $F$ of characteristic $3$ and division symbol $p$-algebras $[\alpha,\beta)_{3,F}$ and $[\alpha,\gamma)_{3,F}$ of degree $3$ over $F$, we prove that if $\alpha \dlog(\beta)\wedge \dlog(\gamma)$ is trivial in the Kato-Milne cohomology group $H_3^3(F)$ then the algebras share a common splitting field which is an inseparable degree 3 extension of either $F$ or a quadratic extension of $F$. In the special case of quadratically closed fields, if $\alpha \dlog(\beta)\wedge \dlog(\gamma)=0$, then they share an inseparable degree 3 extension of $F$.
 \end{abstract}

\begin{keyword}
Kato-Milne Cohomology, Fields of Positive Characteristic, Central Simple Algebras, Division Algebras, Symbol Algebras, $p$-Algebras, Linkage
\MSC[2010] 16K20 (primary); 11E04, 11E81, 19D45 (secondary)
\end{keyword}
\end{frontmatter}

\section{Introduction}

There are several different levels of how closely related two division algebras are to each other.
The closest connection is being isomorphic.
Nonisomorphic algebras can still have something in common, namely a maximal subfield, in which case we say the algebras are ``linked".
If they do not share a maximal subfield, the algebras are non-linked.
For quaternion algebras $Q_1$ and $Q_2$ over a field $F$, the index of $Q_1 \otimes Q_2$ measures the strength of the connection -- being 1 when they are isomorphic, 2 when they are nonisomorphic but linked, and 4 when they are non-linked. The linkage properties of quaternion algebras over a given field show strong connections to the arithmetic properties of the field, such as the $u$-invariant (see \cite{ChapmanDolphin:2017}, \cite{ChapmanDolphinLeep} and \cite{ChapmanMcKinnie:2018} for reference).

The story becomes more complicated for division symbol $p$-algebras of prime degree $p$ over fields $F$ of $\operatorname{char}(F)=p$.
Such algebras admit the structure
$$[\alpha,\beta)_{p,F}=F \langle x,y : x^p-x=\alpha, y^p=\beta, y x y^{-1}=x+1 \rangle$$
for some $\alpha \in F$ and $\beta \in F^\times$ such that $\alpha \not \in \wp(F)=\{\lambda^p-\lambda : \lambda \in F\}$ and $\beta$ is not a norm in the field extension $F[\wp^{-1}(\alpha)]/F$.
For two such algebras, $A_1$ and $A_2$, we say that they are ``cyclically linked" if they share a cyclic degree $p$ extension of $F$, in which case one can write $A_1=[\alpha,\beta_1)_{p,F}$ and $A_2=[\alpha,\beta_2)_{p,F}$ for some $\alpha,\beta_1,\beta_2$. We say the algebras are ``inseparably linked" if they share a purely inseparable degree $p$ extension of $F$, in which case one can write $A_1=[\alpha_1,\beta)_{p,F}$ and $A_2=[\alpha_2,\beta)_{p,F}$ for some $\alpha_1,\alpha_2,\beta$.
In \cite{Chapman:2015} it was proven that inseparable linkage implies cyclic linkage, and examples were provided to demonstrate that the converse is in general false. 

We say that the algebras are totally cyclically (or inseparably) linked if every cyclic (purely inseparable) degree $p$ field extension of $F$ that embeds into one of them, embeds also into the other.
Examples of non-isomorphic totally cyclically linked algebras appear in \cite{ChapmanDolphinLaghribi} for $p=2$, where it is also shown that total cyclic linkage and total inseparable linkage are independent properties that do not imply each other.
See Section \ref{Construction} for the construction of totally cyclically or inseparably linked algebras for arbitrary $p$.
One can outline the story in the following diagram
$$
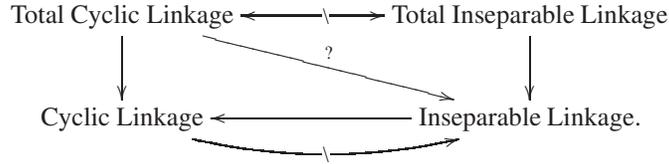
\xymatrix{
\text{Total Cyclic Linkage}\ar@{->}[d]\ar@{->}[drr]^{?}\ar@{<->}[rr]|{\backslash}& &\text{Total Inseparable Linkage}\ar@{->}[d]\\
\text{Cyclic Linkage}\ar@/_3ex/@{->}[rr]|{\backslash}\ar@{<-}[rr]& &\text{Inseparable Linkage.}
}$$
Clearly totally inseparable linkage implies cyclic linkage. The missing part in this puzzle is whether total cyclic linkage implies inseparable linkage.

In this paper we tackle the following two problems:
\begin{enumerate}
\item Find a sufficient condition for cyclically linked division symbol $p$-algebras to be inseparably linked.
\item Does total cyclic linkage imply inseparable linkage?
\end{enumerate}
We provide answers to both problems in the case of $p=3$ for quadratically closed fields.
Note that answers to both problems exist in the literature for $p=2$: Problem 1 was answered in \cite{ElduqueVilla:2005} and \cite{ChapmanGilatVishne:2017} and Problem 2 in \cite{ChapmanDolphin:2018}.

\section{Kato-Milne Cohomology}

Assume $F$ is a field of characteristic $p>0$. For $n \geq 0$, the Kato-Milne Cohomology group $H_p^{n+1}(F)$ is defined to be the cokernel of the Artin-Schreier map
$$\wp : \Omega_F^n \rightarrow \Omega_F^n/\text{d} \Omega_F^{n-1}$$
$$\alpha \dlog(\beta_1) \wedge \dots \wedge \dlog(\beta_n) \mapsto (\alpha^p-\alpha) \dlog(\beta_1) \wedge \dots \wedge \dlog(\beta_n).$$
In particular, $H_p^1(F)=F/\wp(F)$. 
It is known that 
$H_p^2(F) \cong {_pBr(F)}$ (see \cite[Theorem 9.2.4]{GS}).
The isomorphism is given by the map
$$\alpha \dlog(\beta) \mapsto [\alpha,\beta)_{p,F}.$$
The elements of $H_p^3(F)$ seem to be connected to linkage properties. For $p=2$, the quaternion algebras $[\alpha,\beta)_{2,F}$ and $[\alpha,\gamma)_{2,F}$ are inseparably linked if and only if $\alpha \dlog(\beta) \wedge \dlog(\gamma)$ is trivial in $H_2^3(F)$ (see \cite[Theorem 5.3]{ChapmanGilatVishne:2017}).
The goal of this paper is to provide an analogous result for $p=3$.
The following known result will be useful in our study:
\begin{thm}[{\cite[Th\'eor\`eme 6]{Gille:2000}}]
The class of $\alpha \dlog(\beta)\wedge \dlog(\gamma)$ is trivial in $H_p^3(F)$ if and only if $\gamma$ is the norm of an element in the algebra $[\alpha,\beta)_{p,F}$.
\end{thm}

\begin{rem}
The result from \cite{ChapmanGilatVishne:2017} that the division quaternion algebras $[\alpha,\beta)_{2,F}$ and $[\alpha,\gamma)_{2,F}$ over a field $F$ of characteristic 2 are inseparably linked if and only if $\alpha \dlog(\beta)\wedge \dlog(\gamma)=0$ can be concluded directly from this theorem:
they are inseparably linked if and only if the pure part of the underlying Albert form is isotropic, i.e., if $\langle 1 \rangle \perp \beta [1,\alpha] \perp \gamma [1,\alpha]$ is isotropic. This form is isotropic if and only if there is a nonzero element $\lambda \in L=F[\wp^{-1}(\alpha)]$ such that $\norm_{L/F}(\lambda)\gamma$ is represented by $\langle 1 \rangle \perp \beta [1,\alpha]$. Note that the form $[1,\alpha]$ is the norm form from $L$ to $F$. Since the norm form is multiplicative, there exists $\lambda \in L^\times$ such that $\norm_{L/F}(\lambda)\gamma$ is represented by $\langle 1 \rangle \perp \beta [1,\alpha]$ if and only if $\gamma$ is represented by $[1,\alpha] \perp \beta [1,\alpha]$, which is the norm form of $[\alpha,\beta)_{2,F}$.
\end{rem}

\section{Construction of Totally Cyclically and Inseparably Linked Algebras}\label{Construction}

In this section we show how to construct totally cyclically or inseparably linked symbol $p$-algebras which do not generate the same subgroup of the Brauer group. We follow the construction from \cite{Tikhonov:2016}, and make use of the theory of valued division algebras whose chief reference is \cite{TignolWadsworth:2015}. Recall that given a division algebra $D$ over a Henselian valued field $F$, the valuation extends uniquely from $F$ to $D$. We denote the value group by $\Gamma_D$ and the residue algebra by $\overline{D}$, and the following ``fundamental inequality" is satisfied: 
$[\overline{D}:\overline{F}] \cdot [\Gamma_D:\Gamma_F] \leq [D:F].$
The algebra is called ``defectless" if the inequality above is an equality, and ``unramified" if $\Gamma_D=\Gamma_F$.
\begin{thm}[{\cite[Theorem 1]{Morandi:1989}}]\label{Morandi}
Suppose $F$ is a Henselian valued field, $D$ and $E$ are division algebras over $F$ such that 
\begin{enumerate}
\item $D$ is defectless,
\item $\overline{D} \otimes \overline{E}$ is a division algebra, and
\item $\Gamma_D \cap \Gamma_E=\Gamma_F$.
\end{enumerate}
Then $D \otimes E$ is a division algebra.
\end{thm}

As a result of this theorem we obtain:
\begin{lem}\label{MorandiCor}
Let $A$ be a division algebra over a field $F$ of $\operatorname{char}(F)=p$. Write $A_{F(x)}$ for $A \otimes F(x)$ where $F(x)$ is the function field in one variable over $F$.
\begin{itemize}
\item If $K=F[\wp^{-1}(\gamma)]$ is a cyclic field extension of $F$ of degree $p$ for which $A_K$ remains a division algebra, then $A_{F(x)} \otimes [\gamma,x)_{p,F(x)}$ is a division algebra.
\item If $L=F[\sqrt[p]{\delta}]$ is a purely inseparable field extension of $F$ of degree $p$ for which $A_L$ remains a division algebra, then $A_{F(x)} \otimes [x^{-1},\delta)_{p,F(x)}$ is a division algebra.
\end{itemize}
\end{lem}

\begin{proof}
Consider the $x$-adic valuation on $F(x)$.
We keep the notation from Theorem \ref{MorandiCor} and put $D=A_{F(x)}$. Then $D$ has $\overline{D}=A$ as its residue algebra, and it is therefore defectless and unramified.
We put $E=[\gamma,x)_{p,F(x)}$ in 1 and $E=[x^{-1},\delta)_{p,F(x)}$ in 2, whose residue algebra $\overline{E}$ is $K$ or $L$ respectively, and so $\overline{D} \otimes \overline{E}$ is a division algebra.
The condition $\Gamma_D \cap \Gamma_E=\Gamma_F$ also holds true, because $D$ is unramified. Therefore $D\otimes E$ is a division algebra by Theorem \ref{Morandi}.
\end{proof}

We are now ready to explain how nonisomrphic totally cyclically or inseparably linked algebras are constructed:
\begin{thm}\label{TotalConst}
Consider two cyclically linked division symbol $p$-algebras $A$ and $B$ of degree $p$ over $F$ where $B$ is not isomorphic to $A^{\otimes t}$ for any integer $t$.
\begin{itemize}
\item There exists a field extension $M$ over which $A_M$ and $B_M$ remain division algebras, $B_M$ is not isomorphic to $A_M^{\otimes t}$ for any integer $t$, and $A_M$ and $B_M$ are totally cyclically linked.
\item There exists a field extension $T$ over which $A_T$ and $B_T$ remain division algebras, $B_T$ is not isomorphic to $A_T^{\otimes t}$ for any integer $t$, and $A_T$ and $B_T$ are totally inseparably linked.
\end{itemize} 
\end{thm}

\begin{proof}
Suppose there is a cyclic degree $p$ field extension $K=F[\wp^{-1}(\gamma)]$ of $F$ which is a subfield of $B$ but not of $A$.
By Lemma \ref{MorandiCor} (1), $C=A_{F(x)}^{\text{op}} \otimes_{F(x)} [\gamma,x)_{p,F(x)}$ is a division algebra. Let $R$ be the function field of its Severi-Brauer variety.
Then $A_R$ is isomorphic to $[\gamma,x)_{p,R}$, and in particular has $K \otimes_F R$ as a splitting field.
Now take $D$ to be an arbitrary central simple algebra of degree $p$ over $F$. Suppose that $D$ is split by $R$.
Then $D_{F(x)}$ is Brauer equivalent to $C^{\otimes i}$ for some $i \in \{1,\dots,p\}$ by \cite[Theorem 13.10]{Saltman:1998}. If $i<p$ then $C^{\otimes i}$ ramifies at the $x$-adic valuation but $D_{F(x)}$ is unramified, and so $D_{F(x)}$ is not Brauer equivalent to $C^{\otimes i}$. Since the exponent of $D$ is $p$, it cannot be Brauer equivalent to $C^{\otimes p}$ either, contradiction. Therefore $D$ is not split by $R$.
By plugging in $D$ the symbol $p$-algebra equivalent to $A^{\otimes t} \otimes B^{\text{op}}$ for any integer $t$, we obtain that $B_R$ is not isomorphic to $A_R^{\otimes t}$ for any $t$.
Using this construction inductively we obtain a field extension $M$ of $F$ over which $A_M$ and $B_M$ remain division algebras, $B_M$ is not isomorphic to $A_M^{\otimes t}$ for any integer $t$, and $A_M$ and $B_M$ are totally cyclically linked.

Now suppose there is a purely inseparable degree $p$ field extension $L=F[\sqrt[p]{\delta}]$ of $F$ which is a subfield of $B$ but not of $A$.
By Lemma \ref{MorandiCor} (2), $C=A_{F(x)}^{\text{op}} \otimes_{F(x)} [x^{-1},\delta)_{p,F(x)}$ is a division algebra.
Let $R$ be the function field of its Severi-Brauer variety.
Then $A_R$ is isomorphic to $[x^{-1},\delta)_{p,R}$, and in particular has $L \otimes_F R$ as a splitting field.
Every central simple algebra $D$ of degree $p$ over $F$ remains a division algebra over $R$ for the same reason as in the previous case, and so $B_R$ is not isomorphic to $A_R^{\otimes t}$ for any $t$.
Using this construction inductively we obtain a field extension $T$ of $F$ over which $A_T$ and $B_T$ remain division algebras, $B_T$ is not isomorphic to $A_M^{\otimes t}$ for any integer $t$, and $A_T$ and $B_T$ are totally inseparably linked.
\end{proof}

To obtain an explicit example, one can start with $A=[1,\alpha)_{p,F}$ and $B=[1,\beta)_{p,F}$ over the function field $F=\mathbb{F}_p(\alpha,\beta)$ in two algebraically independent variables over the finite field $\mathbb{F}_p$ in $p$ elements, and apply Theorem \ref{TotalConst}. 

\section{Linkage of Symbol $p$-Algebras of Degree 3}

Here we prove the main results of the paper which deal with sufficient conditions for inseparable linkage.

\begin{thm}\label{norm6}
Suppose $F$ is a field of $\operatorname{char}(F)=3$.
Let $[\alpha,\beta)_{3,F}$ be a division symbol algebra, and $\gamma \in F^\times \setminus (F^\times)^3$.
Then $\alpha \dlog(\beta)\wedge \dlog(\gamma)$ is trivial in $H_3^3(F)$ if and only if there exists $\lambda \in L=E[\wp^{-1}(\alpha)]$ such that $[\alpha,\beta)_{3,E}$ contains the purely inseparable subfield $E[z : z^3=\norm_{L/E}(\lambda) \gamma]$ where $E$ is either $F$ or a quadratic extension of $F$.
\end{thm}

\begin{proof}
Suppose that $\alpha \dlog(\beta)\wedge d\log(\gamma)$ is trivial in $H_3^3(F)$. Then $\gamma$ is a norm of an element $r \in A=[\alpha,\beta)_{3,F}$.
Recall that $A$ is a central simple algebra of degree 3 over $F$, and therefore there exist three characteristic forms $\tr,\s,\norm : A \rightarrow F$ of homogeneous degrees 1,2 and 3, respectively, such that each $t\in A$ satisfies 
$$t^3-\tr(t)t^2-\s(t)t-\norm(t)=0.$$
Now $K=F[x : x^3-x=\alpha]$ is a cyclic subfield of $A$. Consider the maps $f_1$ and $f_2$ from $K$ to $F$ defined by
$$f_1(\lambda)=\tr(\lambda r), \ \text{and} \ f_2(\lambda)=\sigma(\lambda r).$$
The equation $f_1(\lambda)=0$ is a linear equation on $K$.
The space of solutions is either a two-dimensional $F$-subspace of $K$ or the entire field $K$ (when the equation is trivial). In either case, the space of solutions contains a two-dimensional $F$-subspace $V$ of $K$.
The restriction of $f_2$ to $V$ is a two-dimensional quadratic form over $F$. This form has a root in $E$ where $E$ is either $F$ or a quadratic extension of $F$, which means that there exists a nonzero $\lambda \in L=E \otimes_F K$ for which $\tr(\lambda r)=\sigma(\lambda r)=0$.
The element $z=\lambda r$ therefore generates a purely inseparable field extension of $E$ inside $A \otimes E$. Its norm is $\norm_{L/E}(\lambda) \cdot \norm(r)=\norm_{L/E}(\lambda) \gamma$.

In the opposite direction, suppose there exists $\lambda \in L=E[\wp^{-1}(\alpha)]$ such that $[\alpha,\beta)_{3,E}$ contains the purely inseparable subfield $E[z : z^3=\norm_{L/E}(\lambda) \gamma]$ where $E$ is either $F$ or a quadratic extension of $F$.
Then the algebra can be written as $[\delta,\norm_{L/E}(\lambda)\gamma)_{3,E}$ for some $\delta \in E$. Then a straight-forward computation shows that $\alpha \dlog(\beta)\wedge \dlog(\gamma)=0$:
\begin{eqnarray*}
\alpha\dlog(\beta)\wedge \dlog(\gamma)&=&-\alpha \dlog(\gamma)\wedge \dlog(\beta)=\\-\alpha \dlog(\norm_{L/E}(\lambda)\gamma)\wedge \dlog(\beta)&=&\alpha \dlog(\beta) 
\wedge \dlog(\norm_{L/E}(\lambda)\gamma)=
\\ \delta \dlog(\norm_{L/E}(\lambda)\gamma)\wedge \dlog(\norm_{L/E}(\lambda)\gamma)&=&0.
\end{eqnarray*}
If $E=F$ this means exactly that $\alpha\dlog(\beta)\wedge \dlog(\gamma)$ is trivial in $H_3^3(F)$.
If $E$ is a quadratic extension of $F$, it means that the restriction of $\alpha\dlog(\beta)\wedge \dlog(\gamma)$ to $E$ is trivial in $H_3^3(E)$.
However, the corestriction back to $F$ of this restriction to $E$ is $2\alpha\dlog(\beta)\wedge \dlog(\gamma)$, which is trivial in $H_3^3(F)$ if and only if $\alpha\dlog(\beta)\wedge \dlog(\gamma)$ is trivial.
Since the restriction is trivial, the coresstriction is trivial too. Consequently, $2\alpha\dlog(\beta)\wedge \dlog(\gamma)=0$ in $H_3^3(F)$, and so $2\alpha\dlog(\beta)\wedge \dlog(\gamma)=0$ in $H_3^3(F)$.
\end{proof}

\begin{thm}\label{norm}
Suppose $F$ is a quadratically closed field of $\operatorname{char}(F)=3$.
Let $[\alpha,\beta)_{3,F}$ be a division symbol algebra, and $\gamma \in F^\times \setminus (F^\times)^3$.
Then $\alpha \dlog(\beta)\wedge \dlog(\gamma)$ is trivial in $H_3^3(F)$ if and only if there exists $\lambda \in L=F[\wp^{-1}(\alpha)]$ such that $[\alpha,\beta)_{3,F}$ contains the purely inseparable subfield $F[z : z^3=\norm_{L/F}(\lambda) \gamma]$.
\end{thm}

\begin{proof}
The statement follows from Theorem \ref{norm6}, given the assumption that $F$ is quadratically closed, which means that the field $E$ in the proof of that theorem cannot be a quadratic extension of $F$, and so $E=F$.
\end{proof}

The implication on inseparable linkage is immediate:
\begin{thm}\label{ins6}
Suppose $F$ is a field of $\operatorname{char}(F)=3$.
Let $[\alpha,\beta)_{3,F}$ and $[\alpha,\gamma)_{3,F}$ be division symbol algebras.
If $\alpha \dlog(\beta)\wedge \dlog(\gamma)$ is trivial in $H_3^3(F)$ then the algebras share a splitting field which is a degree 3 inseparable extension of either $F$ or a quadratic extension of $F$.
\end{thm}

\begin{proof}
By Theorem \ref{norm6}, $E[z : z^3=\norm_{L/E}(\lambda) \gamma]$ is a splitting field of $[\alpha,\beta)_{3,F}$ for some nonzero $\lambda \in L=E[\wp^{-1}(\alpha)]$. Therefore $[\alpha,\beta)_{3,E}=[\delta,\norm_{L/E}(\lambda) \gamma)_{3,E}$ for some $\delta \in E$, where $E$ is either $F$ or a quadratic extension of $F$.
At the other end, $[\alpha,\gamma)_{3,E}=[\alpha,\norm_{L/E}(\lambda)\gamma)_{3,E}$, and so the algebras $[\alpha,\beta)_{3,F}$ and $[\alpha,\gamma)_{3,F}$ share $L$ as a splitting field.
\end{proof}

\begin{thm}\label{ins}
Suppose $F$ is a quadratically closed field of $\operatorname{char}(F)=3$.
Let $[\alpha,\beta)_{3,F}$ and $[\alpha,\gamma)_{3,F}$ be division symbol algebras.
If $\alpha \dlog(\beta)\wedge \dlog(\gamma)$ is trivial in $H_3^3(F)$ then the algebras are inseparably linked.
\end{thm}

\begin{proof}
Since $F$ is quadratically closed, the field $L$ from the proof of Theorem \ref{norm6} is a degree 3 inseparable extension of $F$. The statement then follows.
\end{proof}

\begin{cor}
When $F$ is a quadratically closed field of characteristic 3 with trivial $H_3^3(F)$, cyclic linkage and inseparable linkage are the same for division symbol $p$-algebras of degree 3 over $F$.
\end{cor}
Such fields can be easily constructed. For example, take the algebraic closure $F_0$ of any field of characteristic $3$ (such as $\mathbb{F}_3$), and look at the quadratic closure $F$ of the function field $F_0(\alpha,\beta)$ in two algebraically independent variables $\alpha$ and $\beta$. The field $F$ is quadratically closed, of characteristic 3, has trivial $H_3^3(F)$, and admits different division symbol $p$-algebras of degree 3.

The following known result enables us to connect this to total cyclic linkage:
\begin{prop}[{\cite[Corollary 3.3]{ChapmanDolphin:2018}}]\label{Total}
For fields $F$ of positive characteristic $p$, if $[\alpha,\beta)_{p,F}$ and $[\alpha,\gamma)_{p,F}$ are totally cyclically linked then $\alpha \dlog(\beta)\wedge \dlog(\gamma)$ is trivial in $H_p^3(F)$.
\end{prop}
\begin{cor}
When $F$ is a quadratically closed field of characteristic 3, every two totally cyclically linked division symbol $p$-algebras of degree 3 over $F$ are inseparably linked.
\end{cor}
\begin{proof}
Consider two totally cyclically linked division symbol $p$-algebras of degree 3 over $F$.
Since they are cyclically linked, one can write them as $[\alpha,\beta)_{3,F}$ and $[\alpha,\gamma)_{3,F}$ for appropriate $\alpha,\beta,\gamma$. By Proposition \ref{Total}, $\alpha \dlog(\beta)\wedge\dlog(\gamma)$ is trivial in $H_3^3(F)$. Then by Theorem \ref{ins}, the algebras are inseparably linked.
\end{proof}

\section*{Acknowledgements}

The author thanks the anonymous referee for the careful reading of the submitted manuscript and the useful comments.

\end{document}